\documentclass[12pt,a4paper]{article}

\usepackage{hyperref}
\usepackage{csquotes}
\usepackage[bbgreekl]{mathbbol}
\usepackage{amsfonts}
\usepackage{amsmath}
\usepackage{amssymb}
\DeclareSymbolFontAlphabet{\mathbb}{AMSb}
\DeclareSymbolFontAlphabet{\mathbbl}{bbold}

\usepackage[greek,english,ngerman]{babel}
\usepackage{graphicx}
\usepackage{amsmath}
\usepackage{verbatim}
\usepackage{color}
\usepackage{amstext}
\usepackage{latexsym}
\usepackage[small,it]{caption}
\usepackage{dsfont}
\usepackage{enumerate,enumitem}
\usepackage{mathtools}
\usepackage{indentfirst}
\usepackage{subcaption}
\usepackage{eurosym}
\usepackage{stmaryrd}

\usepackage[h]{esvect} 

\usepackage{bm}
\usepackage{amsthm}

\usepackage{url}

\makeatletter
\def\@endtheorem{\endtrivlist}
\makeatother
\theoremstyle{plain}
\newtheorem*{theorem*}{\protect\TheoremName}
\newtheorem*{acknowledgement*}{\protect\AcknowledgmentName}
\newtheorem*{algorithm*}{\protect\AlgorithmName}
\newtheorem*{assumption*}{\protect\AssumptionName}
\newtheorem*{assumptions*}{\protect\AssumptionsName}
\newtheorem*{axiom*}{\protect\AxiomName}
\newtheorem*{condition*}{\protect\ConditionName}
\newtheorem*{conditions*}{\protect\ConditionsName}
\newtheorem*{case*}{\protect\CaseName}
\newtheorem*{claim*}{\protect\ClaimName}
\newtheorem*{conclusion*}{\protect\ConclusionName}
\newtheorem*{conjecture*}{\protect\ConjectureName}
\newtheorem*{corollary*}{\protect\CorollaryName}
\newtheorem*{criterion*}{\protect\CriterionName}
\newtheorem*{definition*}{\protect\DefinitionName}
\newtheorem*{lemma*}{\protect\LemmaName}
\newtheorem*{notation*}{\protect\NotationName}
\newtheorem*{problem*}{\protect\ProblemName}
\newtheorem*{proposition*}{\protect\PropositionName}
\newtheorem*{solution*}{\protect\SolutionName}
\newtheorem*{summary*}{\protect\SummaryName}

\newtheorem{theorem}{\protect\TheoremName}[section]

\newtheorem{corollary}[theorem]{\protect\CorollaryName}

\newtheorem{definition}[theorem]{\protect\DefinitionName}
\newtheorem{lemma}[theorem]{\protect\LemmaName}

\theoremstyle{definition}
\newtheorem*{example*}{\protect\ExampleName}
\newtheorem*{exercise*}{\protect\ExerciseName}
\newtheorem*{remark*}{\protect\RemarkName}
\newtheorem{example}[theorem]{\protect\ExampleName}

\newtheorem{remark}[theorem]{\protect\RemarkName}
\newtheorem*{question*}{\protect\QuestionName}
\newtheorem*{questions*}{\protect\QuestionsName}

\newcommand{\TheoremName}{}
\newcommand{\AcknowledgmentName}{}
\newcommand{\AlgorithmName}{}
\newcommand{\AssumptionName}{}
\newcommand{\AssumptionsName}{}
\newcommand{\QuestionName}{}
\newcommand{\QuestionsName}{}
\newcommand{\AxiomName}{}
\newcommand{\ConditionName}{}
\newcommand{\ConditionsName}{}
\newcommand{\CaseName}{}
\newcommand{\ClaimName}{}
\newcommand{\ConjectureName}{}
\newcommand{\ConclusionName}{}
\newcommand{\CorollaryName}{}
\newcommand{\CriterionName}{}
\newcommand{\DefinitionName}{}
\newcommand{\LemmaName}{}
\newcommand{\NotationName}{}
\newcommand{\ProblemName}{}
\newcommand{\PropositionName}{}
\newcommand{\SolutionName}{}
\newcommand{\SummaryName}{}
\newcommand{\ExampleName}{}
\newcommand{\RemarkName}{}
\newcommand{\ExerciseName}{}
\addto\captionsenglish{%
  \renewcommand{\TheoremName}{Theorem}
  \renewcommand{\AcknowledgmentName}{Acknowledgments}
  \renewcommand{\AlgorithmName}{Algorithm}
  \renewcommand{\AssumptionName}{Assumption}
  \renewcommand{\AssumptionsName}{Assumptions}
  \renewcommand{\QuestionName}{Question}
  \renewcommand{\QuestionsName}{Questions}
  \renewcommand{\AxiomName}{Axiom}
  \renewcommand{\ConditionName}{Condition}
  \renewcommand{\ConditionsName}{Conditions}
  \renewcommand{\CaseName}{Case}
  \renewcommand{\ClaimName}{Claim}
  \renewcommand{\ConjectureName}{Conjecture}
  \renewcommand{\ConclusionName}{Conclusion}
  \renewcommand{\CorollaryName}{Corollary}
  \renewcommand{\CriterionName}{Criterion}
  \renewcommand{\DefinitionName}{Definition}
  \renewcommand{\LemmaName}{Lemma}
  \renewcommand{\NotationName}{Notation}
  \renewcommand{\ProblemName}{Problem}
  \renewcommand{\PropositionName}{Proposition}
  \renewcommand{\SolutionName}{Solution}
  \renewcommand{\SummaryName}{Summary}
  \renewcommand{\ExampleName}{Example}%
  \renewcommand{\ExerciseName}{Exercise}%
  \renewcommand{\RemarkName}{Remark}%
}
\addto\captionsngerman{%
  \renewcommand{\TheoremName}{Theorem}
  \renewcommand{\AcknowledgmentName}{Anerkennungen}
  \renewcommand{\AlgorithmName}{Algorithmus}
  \renewcommand{\AssumptionName}{Voraussetzung}
  \renewcommand{\AssumptionsName}{Voraussetzungen}
  \renewcommand{\AxiomName}{Axiom}
  \renewcommand{\ConditionName}{Bedingung}
  \renewcommand{\ConditionsName}{Bedingungen}
  \renewcommand{\CaseName}{Fall}
  \renewcommand{\ClaimName}{Behauptung}
  \renewcommand{\ConjectureName}{Vermutung}
  \renewcommand{\ConclusionName}{Folgerung}
  \renewcommand{\CorollaryName}{Korollar}
  \renewcommand{\CriterionName}{Kriterium}
  \renewcommand{\DefinitionName}{Definition}
  \renewcommand{\LemmaName}{Lemma}
  \renewcommand{\NotationName}{Notation}
  \renewcommand{\ProblemName}{Problem}
  \renewcommand{\PropositionName}{Satz}
  \renewcommand{\SolutionName}{L\"osung}
  \renewcommand{\SummaryName}{Zusammenfassung}
  \renewcommand{\ExampleName}{Beispiel}%
  \renewcommand{\ExerciseName}{\"Ubung}%
  \renewcommand{\RemarkName}{Bemerkung}%
}
\addto\captionsgreek{%
  \renewcommand{\TheoremName}{Je\'wrhma}
  \renewcommand{\AcknowledgmentName}{Euqarist\'ies}
  \renewcommand{\AlgorithmName}{Alg\'orijmos}
  \renewcommand{\AssumptionName}{Up\'ojesh}
  \renewcommand{\AssumptionsName}{Upoj\'eseis}
  \renewcommand{\AxiomName}{Ax\'iwma}
  \renewcommand{\ConditionName}{Pro\"up\'ojesh}
  \renewcommand{\ConditionsName}{Pro\"upoj\'eseis}
  \renewcommand{\CaseName}{Per\'iptwsh}
  \renewcommand{\ClaimName}{Isqurism\'os}
  \renewcommand{\ConjectureName}{Eikas\'ia}
  \renewcommand{\ConclusionName}{Sump\'erasma}
  \renewcommand{\CorollaryName}{P\'orisma}
  \renewcommand{\CriterionName}{Krit\'hrio}
  \renewcommand{\DefinitionName}{Orism\'os}
  \renewcommand{\LemmaName}{L\'hmma}
  \renewcommand{\NotationName}{Shmeiograf\'ia}
  \renewcommand{\ProblemName}{Pr\'oblhma}
  \renewcommand{\PropositionName}{Pr\'otash}
  \renewcommand{\SolutionName}{L\'ush}
  \renewcommand{\SummaryName}{Per\'ilhyh}
  \renewcommand{\ExampleName}{Par\'adeigma}%
  \renewcommand{\ExerciseName}{\'Askhsh}%
  \renewcommand{\RemarkName}{Parat\'hrhsh}%
}

\makeatletter
\renewenvironment{proof}[1][\proofname]{\par
\pushQED{\hfill$\blacksquare$}%
\normalfont \topsep6\p@\@plus6\p@\relax
\trivlist
\item\relax
{\bfseries
#1\@addpunct{.}}\hspace\labelsep\ignorespaces
}{%
\popQED\endtrivlist\@endpefalse
}
\makeatother

\newcommand{\strongly}{%
  \mathrel{
    \resizebox{1.4em}{0.88ex}{$\rightarrow$}
}}

\newcommand{\wstar}{%
  \mathrel{\vbox{\offinterlineskip\ialign{%
    \hfil##\hfil\cr
    $\scriptstyle*\,$\cr
    \noalign{\kern 0.12ex}
    \resizebox{1.4em}{0.88ex}{$\rightharpoonup$}\cr
}}}}

\newcommand{\twoweakly}{%
  \mathrel{\vbox{\offinterlineskip\ialign{%
    \hfil##\hfil\cr
    $\scriptstyle2\,$\cr
    \noalign{\kern 0.12ex}
    \resizebox{1.4em}{0.88ex}{$\rightharpoonup$}\cr
}}}}

\newcommand{\twostrongly}{%
  \mathrel{\vbox{\offinterlineskip\ialign{%
    \hfil##\hfil\cr
    $\scriptstyle2\,$\cr
    \noalign{\kern 0.12ex}
    \resizebox{1.4em}{0.88ex}{$\rightarrow$}\cr
}}}}

\def\XXint#1#2#3{{\setbox0=\hbox{$#1{#2#3}{\int}$ }\hspace{0.17em}
\vcenter{\hbox{$#2#3$ }}\kern-.585\wd0}}

\DeclareMathOperator*{\clconv}{\overline{conv}}

\DeclareMathOperator{\adj}{adj}

\newcommand{\lnorm}{\left\|}
\newcommand{\rnorm}{\right\|}



\makeatletter
\newcommand{\defeq}{\mathrel{\hspace{-0.1ex}\mathrel{\rlap{\raisebox{0.3ex}{$\m@th\cdot$}}\raisebox{-0.3ex}{$\m@th\cdot$}}\hspace{-0.73ex}\resizebox{0.55em}{0.84ex}{$=$}\hspace{0.67ex}}\hspace{-0.25em}}
\makeatother
\makeatletter
\newcommand{\eqdef}{\hspace{-0.25em}\mathrel{\hspace{0.67ex}\resizebox{0.55em}{0.84ex}{$=$}\hspace{-0.73ex}\mathrel{\rlap{\raisebox{0.3ex}{$\m@th\cdot$}}\raisebox{-0.3ex}{$\m@th\cdot$}}\hspace{-0.1ex}}}
\makeatother
\usepackage{tikz}


\usepackage[colorinlistoftodos]{todonotes}
\usepackage[author={A.S.},color=yellow]{pdfcomment}
\DTMsetregional 

\newcommand{\refpart}[2]{\hyperref[#1]{\ref*{#1}--\textit{#2.}}}

\begin{document}
\selectlanguage{english}

\title{On the existence of symmetric minimizers
}
\author{Athanasios Stylianou\thanks{Institut f\"{u}r Mathematik, Universit\"{a}t Kassel, 34132 Kassel, Germany}
}
\maketitle


\begin{abstract}
 In this note we revisit a less known symmetrization method for functions with respect to a topological group, which we call $G$-averaging. We note that, although quite non-technical in nature, this method yields $G$-invariant minimizers of functionals satisfying some relaxed convexity properties. We give an abstract theorem and show how it can be applied to the $p$-Laplace and polyharmonic Poisson problem in order to construct symmetric solutions. We also pose some open problems and explore further possibilities where the method of $G$-averaging could be applied to.
\end{abstract}

\section{Introduction}
Identifying symmetries of problems has always been of importance in the deeper understanding of mathematical and physical problems. As Kawohl wrote in \cite{Kawohl1998}: ``Many problems in analysis appear symmetric, yet in fact their solutions are sometimes nonsymmetric.'' It has been a subject of thorough study in the theory of partial differential equations, to find conditions that imply the inheritance of the symmetries of the problem to its solutions. For example Dacorogna \textit{et al} in \cite{DacorognaGangboEtAl1992} and subsequently Belloni and Kawohl in \cite{BelloniKawohl1999} studied symmetry properties of minimizers of problems related to best Sobolev constants and isoperimetric problems in crystallography. Symmetries of systems of semilinear equations in the context of criticality, were studied by Bozhkov and Mitidieri in \cite{BozhkovMitidieri2007}, using Poho\v zaev's and Noether's identities and Lie group theory.  When it comes to quasilinear equations, D'Ambrosio \textit{et al} in \cite{DAmbrosioFarinaEtAl2013} studied the symmetry properties of distributional solutions. Lastly, Kr\"omer studied the radial symmetries of non-convex functionals in \cite{Kroemer2008}.

A number of methods exist in order to prove inheritance of symmetries, many of them are described in \cite{Kawohl1998} and in the more recent survey \cite{Weth2010}. Here we will try to give a short description of the available tools, that have been so far developed an implemented; this list of references is by far not exhaustive.

Maybe one of the most widely developed one is Alexandrov's moving plane method \cite{Alexandrov1962}, adapted to pdes by Serrin in \cite{Serrin1971} and further used by Gidas, Ni and Nirenberg in \cite{GidasNiNirenberg1979,GidasNiNirenberg1981} to prove that positive solutions to the semilinear Poisson problem in a ball, are radial. This method has been further refined by Cl\'ement and Sweers in \cite{ClementSweers1989} to include subsolutions and by Kawohl and Sweers in \cite{KawohlSweers2002} to include Steiner-symmetric domains. In the case of exterior domains, we refer to the work \cite{Reichel1997} of Reichel and to the work \cite{Sirakov2001} of Sirakov. For semilinear systems in bounded domains see Troy \cite{Troy1981} and for the unbounded case the work \cite{BuscaSirakov2000} of Busca and Sirakov. For polyharmonic semilinear equations see Gazzola, Grunau and Sweers' monograph \cite{GazzolaGrunauEtAl2010}. For quasilinear equations we refer to Serrin and Zou \cite{SerrinZou1999}. Da Lio and Sirakov in \cite{DaLioSirakov2007} extended the method to viscosity solutions to fully nonlinear equations. Fleckinger and Reichel in \cite{FleckingerReichel2005} proved radiality of global solution branches for problems involving the $p$-Laplacian on balls. Herzog and Reichel treated elliptic systems in ordered Banach spaces in \cite{HerzogReichel2012}. For more details and for an application of to $p$-Laplace equations see also Brezis' survey \cite{Brezis1999}, Pacella and Ramaswamy's survey \cite{PacellaRamaswamy2008} or Fraenkel's monograph \cite{Fraenkel2000}. For more details on the maximum principle and related symmetry results we also refer to Pucci and Serrin's monograph \cite{PucciSerrin2007}. One can also use symmetry methods to prove nonexistence results, as did Poho\v zaev (\cite{Pohozaev1965} or \cite[Chapter III, 1]{Struwe2008} and references therein) or Reichel and Zou in \cite{ReichelZou2000}.

Moreover, a reflection method was considered by Lopes in \cite{Lopes1996_1,Lopes1996_2} and further developed by Mari\c{s} in \cite{Maris2009} in order to prove symmetry of constrained minimizers that possess either some continuation properties or increased regularity. Later works using this method include the one by Jeanjean and Squassina \cite{JeanjeanSquassina2009} and references therein.

Another method consists in the so-called technique of symmetrization and symmetric rearrangements and foliations; see for example Kawohl's monograph \cite{Kawohl1985} and the surveys of Brock \cite{Brock2007} and Pacella and Ramaswamy \cite{PacellaRamaswamy2008}. Symmetrization techniques can also find application in minimax problems, as Van Schaftingen proved in \cite{VanSchaftingen2005}. Concerning Schwarz symmetrization, we would also like to mention Hajaiej and Kr\"omer's recent work \cite{HajaiejKroemer2012} and the references therein.

All the above methods are quite technical and often make assumptions which are similar in nature. However, in some cases, symmetry can be obtained via short and elegant arguments. For example, uniqueness of solutions to a symmetric problem leads directly to invariance; otherwise any symmetric transformation of a solution will also be one. Or, that minimizers of strictly convex functionals must inherit the symmetry of the problem as shown by Montenegro and Valdinoci in \cite{MontenegroValdinoci2012}.

Finally we would also like to refer to the work of Pacella and Weth \cite{PacellaWeth2007} for proving symmetry results via Morse theory, and to the work of Jarohs and Weth \cite{JarohsWeth2016} for nonlocal problems.

The present article is of a similar nature. We consider an averaging type symmetrization of a function and its properties when it comes to inheritance of symmetry. This idea was firstly used by Nicolaescu in \cite{Nicolaescu1988} in order to prove radiality of optimal controls for increasing cost functionals controlled by a semilinear Poisson state equation. In \cite{Kroemer2008}, Kr\"omer applied it also to prove radiality of minimizers. It was also used by Struwe in his monograph \cite[3.3 Remark, p.82]{Struwe2008} to construct symmetric pseudo-gradient flows for functionals on Banach spaces. Up to our knowledge, it has not been used in its full generality to prove inheritance of symmetry to solutions, despite  its compact and elegant form. (The method in the radial cases above will follow from the one presented here if one assumes a mean value theorem for integrals; see Section \ref{mean_value}.) The aim of this note is to explore the possibilities that this method may have, when it comes to finding symmetric solutions.

\section{Existence of invariant minimizers}
We begin by making clear what we mean by the $G$-averaging of an element of a general Banach space. We will later make this more concrete by applying the method on appropriate function spaces.
\begin{definition}
 Let $X$ be a Banach space, $C\subseteq X$ closed and convex and $G$ a compact topological group acting continuously on $C$. Moreover, let $\theta$ denote the Haar-measure on $G$ (the unique left- and right-invariant probability measure with respect to the Borel $\sigma$-algebra on $G$) and $u\in C$. Define the $G$-average of $u$ by
 \begin{equation}\label{ug}
  u_G\defeq \int_G g\cdot u\; d\theta(g),
 \end{equation}
 where the integral is in the sense of Bochner.
\end{definition}
\begin{example}[{\protect\cite[Proposition 4.2]{Nicolaescu1988}}] For $X=C[-1,1]$ we get
 \begin{equation*}
  u_{O(1)}(x)=\frac{1}{2}\big(u(x)+u(-x)\big),
 \end{equation*}
 whereas for $X=C(\overline \Omega)$ with $\Omega\defeq\{z\in \mathbb C: |z|<1\}$ we get
 \begin{equation*}
  u_{O(2)}(z)=\frac{1}{4\,\pi}\int_0^{2\pi}\Big(u\big(e^{i\phi}\,z\big)+u\big(e^{i\phi}\,\overline z\big)\Big)\;d\phi.
 \end{equation*}
\end{example}
Still, it is not directly clear why the $G$-average is well-defined in the abstract setting. This problem is dealt with in the next lemma.
\begin{lemma}\label{inC}
 Let $X$ be a Banach space, $C\subseteq X$ closed and convex and $G$ a compact topological group acting continuously on $C$. Moreover, let $u\in C$. Then the $G$-average of $u$ (defined by \eqref{ug}) satisfies $u_G\in C$.
\end{lemma}
\begin{proof}
 Since $G$ acts on $C$ we have that $g\cdot u\in C$ for all $g\in G$ and since $C$ is closed and convex holds that $\clconv G u\subseteq C$. Moreover, due to the continuity of the $G$-action we get that the real function $g\mapsto \lnorm g\cdot u\rnorm_X$ is continuous. Since $G$ is compact we estimate
 \begin{equation}
  \int_G\lnorm g\cdot u\rnorm_X\;d\theta(g)\leq \max_{g\in G} \lnorm g\cdot u\rnorm_X\, \int_G\;d\theta(g)<\infty,
 \end{equation}
 since $\theta$ is a probability measure and the function $g\mapsto \lnorm g\cdot u\rnorm_X$ attains its maximum in the compact space $G$. Thus the mapping $g\mapsto g\cdot u$ is Bochner integrable. But then \cite[Proposition 1.2.12]{HytoenenVanNeervenEtAl2016} implies that $u_G\in \clconv G u$ and thus $u_G\in C$.
\end{proof}

We can now prove the main abstract result on the inheritance of symmetry from a minimization problem to its minimizers. The importance of the result lies in the fact that we neither impose regularity nor any kind of maximum principle. Our assumptions only include some convexity and invariance. Although $1.$ and $2.$ in the theorem below are partially overlapping, the reason for this formulation is to illustrate different approaches: Jensen's inequality is not indispensable to the proof.
\begin{theorem}\label{th1}
 Let $X$ be a Banach space, $C\subseteq X$ closed and convex and $G$ a compact topological group acting continuously on $C$. Assume that $F:X\strongly \mathbb R$ has a minimizer in $C$ and that it is $G$-invariant with respect to the minimizer, i.e., $F(g\cdot u)=F(u)$ for all $g\in G$, where $u\in C$ is the (local) minimizer. Moreover assume that either
 \begin{enumerate}
  \item $F$ is convex, and either continuous at $u_G$ or lower semi-continuous in $C$, or
  \item $F$ is lower semi-continuous and convex in $\clconv G u$.
 \end{enumerate}
 Then $u_G$ is a $G$-invariant minimizer of $F$ in $C$.
\end{theorem}
\begin{proof}
 First note that $u_G$ is $G$-invariant: due to the (left) invariance of the Haar measure we have for any $\tilde g\in G$ that
 \begin{equation*}
  \tilde g \cdot u_G=\int_G (\tilde g\cdot g)\cdot u\; d\theta(g)=\int_G g\cdot u\; d\theta(g)=u_G.
 \end{equation*}
 We consider the cases separately:
 
%
 \textit{1.} Since $F$ is a convex functional we can either apply a version of Jensen's inequality for the Bochner integral (\cite[Proposition 1.2.11]{HytoenenVanNeervenEtAl2016}):
 \begin{align*}
  F(u_G)={}&F\Big(\int_G g\cdot u\; d\theta(g)\Big)\leq \int_G F(g\cdot u)\; d\theta(g)\\
  ={}&\int_G F(u)\; d\theta(g)=F(u)\int_G\; d\theta(g)=F(u),
 \end{align*}
 or use the fact that the subdifferential of $F$ at $u_G$ is not empty:
 \begin{equation*}
  F(g\cdot u)\geq F(u_G)+\langle \beta^*,g\cdot u-u_G\rangle_{X^*,X}
 \end{equation*}
 for all $g\in G$, for a nonzero $\beta^*\in \partial F(u)$. Integrating the above in $G$ with respect to the Haar measure we obtain $F(u_G)\leq F(u)$ and, together with Lemma \ref{inC}, that $u_G$ is a minimizer in $C$.
 
 \textit{2.} Since $u_G\in \clconv G u$, there exists a sequence $\{t_n\,g_{1,n}\cdot u+(1-t_n)\,g_{2,n}\cdot u\}_{n\in\mathbb N}$, where $g_{1,n},g_{2,n}\in G$, strongly converging to $u_G$. Due to the semi-continuity and convexity of $F$ we get
 \begin{align*}
  F(u_G)={}& F\Big(\lim_{n\rightarrow\infty} \big(t_n\,g_{1,n}\cdot u+(1-t_n)\,g_{2,n}\cdot u\big)\Big)\\
  \leq{}&\lim_{n\rightarrow\infty} F\big(t_n\,g_{1,n}\cdot u+(1-t_n)\,g_{2,n}\cdot u\big)\\
  \leq{}& \lim_{n\rightarrow\infty}\Big(t_n\,F\big(g_{1,n}\cdot u)+(1-t_n)\,F(g_{2,n}\cdot u\big)\Big)=F(u).
 \end{align*}
 Again, using Lemma \ref{inC}, we infer the minimality of $u_G$.
\end{proof}

Next we apply the above abstract theorem in more precise setting. Namely, we prove the inheritance of symmetry from the domain to weak solutions to the $p$-Laplace Poisson problem, only assuming an decreasing nonlinearity.
\begin{corollary}\label{cor1}
 Let $\Omega\subset \mathbb R^n$ be a bounded Lipschitz domain and let $G$ denote the subgroup of $O(n)$ corresponding to the symmetries of $\Omega$. Assume that $G$ is compact and that the functional $F:W^{1,p}_0(\Omega)\strongly \mathbb R$ defined by
 \begin{equation}\label{eq1}
  F(u)\defeq \int_\Omega\Big(\frac{1}{p}\,|\nabla u|^p-f(u)\Big)\;dx,
 \end{equation}
 with $f:\mathbb R\strongly\mathbb R$ concave and $p>1$, possesses a minimizer. Then $F$ possesses a $G$-invariant minimizer.
\end{corollary}
\begin{proof}
 Define the action of $G$ on $W^{1,p}_0(\Omega)$ by composition (pullback) as
 \begin{equation*}
  (g\cdot u) (x)\defeq u(g\,x)
 \end{equation*}
 and note that the standard norm in $W^{1,p}_0(\Omega)$ is $G$-invariant: Since orthogonal matrices preserve lengths we calculate
 \begin{equation*}
  \big|\nabla\big(u(g\,x)\big)\big|=\big|g^\top\,\nabla u(g\,x)\big|=|\nabla u(g\,x)|
 \end{equation*}
 and keep in mind that $|\det g|=1$. We only need to prove the continuity of the action, which, although similarly elementary in nature, we include for the sake of completeness. Let $v\in C_0^\infty(\Omega)$ and $g_1,g_2\in G$. Since $v$ and its first derivatives are Lipschitz continuous we get that
 \begin{equation*}
  \lnorm g_1\cdot v-g_2\cdot v\rnorm_{1,p}\leq C_v\,d(g_1,g_2),
 \end{equation*}
 for some constant $C_v>0$. Here $d$ denotes the standard (Euclidean) metric in $O(n)$. Let $\varepsilon>0$ and pick any $u_1,u_2\in W^{1,p}_0(\Omega)$ with $\lnorm u_1-u_2\rnorm_{1,p}<\varepsilon/4$, $v\in C_0^\infty(\Omega)$ with $\lnorm u_1-v\rnorm_{1,p}<\varepsilon/4$ and $g_1,g_2\in G$ with $d(g_1,g_2)< \varepsilon/(4\,C_v)$. We then obtain that
 \begin{align*}
  {}&\lnorm g_1\cdot u_1-g_2\cdot u_2\rnorm_{1,p}\leq  \lnorm g_1\cdot u_1-g_2\cdot u_1\rnorm_{1,p}+\lnorm g_2\cdot u_1-g_2\cdot u_2\rnorm_{1,p}\\
   {}&\qquad\qquad= \lnorm g_1\cdot u_1-g_2\cdot u_1\rnorm_{1,p}+\lnorm g_2\cdot u_1-g_2\cdot u_2\rnorm_{1,p}\\  
   {}&\qquad\qquad\leq \lnorm g_1\cdot u_1-g_1\cdot v\rnorm_{1,p}+\lnorm g_1\cdot v-g_2\cdot v\rnorm_{1,p}+\lnorm g_2\cdot v-g_2\cdot u_1\rnorm_{1,p}\\
   {}&\qquad\qquad\qquad+\lnorm g_2\cdot u_1-g_2\cdot u_2\rnorm_{1,p}\\  
   {}&\qquad\qquad=\lnorm u_1-v\rnorm_{1,p}+C_v\,d(g_1,g_2)+\lnorm v-u_1\rnorm_{1,p}+\lnorm u_1-u_2\rnorm_{1,p}\\  
   {}&\qquad\qquad<\frac{\varepsilon}{4}+C_v\,\frac{\varepsilon}{4\,C_v}+\frac{\varepsilon}{4}+\frac{\varepsilon}{4}=\varepsilon,
 \end{align*}
 i.e., the action of $G$ on $W^{1,p}_0(\Omega)$ is continuous. Thus, applying Theorem \ref{th1} we obtain the existence of a $G$-invariant minimizer.
\end{proof}
\begin{remark}
 The group $G$ is supposed to include any possible symmetries of the domain, that is, rotations with respect to a fixed angle and a vector, reflections with respect to any hyperplane and of course rigid motions. As long as all these symmetries form a compact subgroup of $O(n)$ (which is for example the case for any finite subgroup, or for any subgroup corresponding to rotations around any number of vectors, since it will be isomorphic to some $SO(m)$, $m\leq n$), the $G$-average of a minimizer will be symmetric with respect to all of the symmetries of the domain.
\end{remark}
\begin{remark}
 The moving plane method relies strongly on the existence of some kind of comparison principle for solutions. But such results are generally not available for higher order pdes. Still, if one restricts the problem to balls, it is possible to substitute the use of comparison principles by kernel estimates and monotonicity properties of the biharmonic Green function (see \cite[Section 7.1]{GazzolaGrunauEtAl2010}). It is not clear how to extend this to general domains, since the a formula for the Green function is explicitly available only for balls (\cite[Section 1.2]{GazzolaGrunauEtAl2010}). The $G$-averaging method does not rely on such results and yields directly symmetric minimizers, as shown in the next corollary.
\end{remark}
\begin{corollary}
 Let $\Omega\subset \mathbb R^n$ be a bounded Lipschitz domain and let $G$ denote the subgroup of $O(n)$ corresponding to the symmetries of $\Omega$. Assume that $G$ is compact and that the functional $\displaystyle F:H^{2\,m}_0(\Omega)\strongly \mathbb R$ defined by
 \begin{equation}\label{eq2}
  F(u)\defeq \int_\Omega\Big(\frac{1}{2}\,(\Delta^m u)^2-f(u)\Big)\;dx,
 \end{equation}
 with $f:\mathbb R\strongly\mathbb R$ concave and $m\in\mathbb N$, possesses a minimizer. Then $F$ possesses a $G$-invariant minimizer.
\end{corollary}
\begin{proof}
 Define as in the proof of Corollary \ref{cor1} the action of the group via composition and note that it is continuous in $\displaystyle H^{2\,m}_0(\Omega)$. Moreover, since the Laplacian is invariant with respect to orthogonal transformations, so is its $m$-th power. Finally, $F$ is a convex functional so that Theorem \ref{th1} proves the claim.
\end{proof}
\begin{remark}
 Since the method of $G$-averaging works for minimizers, the assumptions for right-hand sides are relaxed compared to the moving plane method. Writing the strong Euler-Lagrange equations for \eqref{eq1} and \eqref{eq2}:
 \begin{equation*}
  \bigg\{
  \begin{aligned}
   -\Delta_p u ={}& f'(u) {}&& \text{ in } \Omega\\
   u={}&0{}&& \text{ on }\partial \Omega
  \end{aligned}\quad\text{ and }\quad
  \bigg\{
  \begin{aligned}   
   \Delta^{2\,m} u ={}& f'(u) {}&& \text{ in } \Omega\\
   \partial^\alpha u|_{\partial\Omega}={}&0{}&& \text{ for }|\alpha|\leq 2\,m-1
  \end{aligned}
 \end{equation*}
remark that we only assume a decreasing $f'$, whereas one would normally assume Lipschitz continuous right-hand sides when applying the moving plane method. Still, there are refinements of the latter assumption; discontinuous nonlinearities satisfying some technical assumptions are treated in \cite[Section 3.2]{Fraenkel2000}.
\end{remark}

\section{Further possible applications of the $G$-averaging}
In what follows we will present some open questions and possible generalizations and applications of the presented method.

\subsection{A mean value theorem in Banach spaces}\label{mean_value} If one assumes that a ``strong'' mean value theorem is true, that is, there exists $g_u\in G$ such that $u_G=g_u\cdot u$ for a minimizer $u$, then we can directly see that
$$F(u_G)=F(g_u\cdot u)=F(u), \text{ and } u=(g_u^{-1}\cdot g_u)\cdot u=g_u^{-1}\cdot u_G=u_G.$$
In general, it is not possible to obtain anything better than the assertion
\begin{equation*}
 \frac{1}{\mu(A)}\int_A f(x)\;d\mu(x)\in\clconv f(A)
\end{equation*}
for a Bochner integrable function $f:\Omega\strongly X$ and a $\mu$-measurable $A\subset \Omega$. One way to prove this is using the Hahn-Banach separation theorem to arrive to a contradiction (see \cite[Proposition 2.1.21]{GasinskiPapageorgiou2006}). Would it be possible to use non-convex separation theorems involving separating functionals in normal cones like the one in \cite{BorweinJofre1998} together with specific connectedness assumptions on the group $G$, to obtain a more accurate result for the $G$-average? Having such, one could relax the convexity assumptions on the functional.

\subsection{Less convexity, more structure} As already mentioned in the introduction, there are methods to prove the symmetry of minimizers that do not assume convexity (but still either some regularity of positivity) for solutions to pdes. Is it possible to obtain a result just by direct substituting the $G$-average into the functional? What is the relation of the $G$-average to other symmetrizations? Is it possible to deal with more general functionals? Answers to these questions will not only enable the study of minimizers, but also of critical points in the spirit of \cite{VanSchaftingen2005}. Taking $G$-averages of a Palais-Smale sequence would directly lead to critical points, since the action of the group $G$ on the underlying space is continuous.

\subsection{Polyconvexity} An interesting application of the $G$-average would be in the context of nonlinear elasticity. This is a vectorial case and as it is well known (see eg. Ciarlet's classical monograph \cite{Ciarlet1988}) that convexity turns out to be quite restrictive: so-called hyperelastic materials are modelled by non-convex energies. That is where the so-called polyconvex functionals are of interest. This notion was firstly introduced by Ball in \cite{Ball1976} and it is a sufficient condition for the weak lower semi-continuity of the energies. Still, the proof of Theorem \ref{th1} does not directly work for polyconvex functionals. We shortly describe the problem and begin with the definition:
 \begin{definition}[{\protect \cite[Definition 5.1]{Dacorogna2008}}]
 A function $f:\mathbb R^{k\times n}\strongly\mathbb R$ is called \emph{polyconvex} if there exists $g:\mathbb R^{\tau(k,n)}\strongly\mathbb R$ convex, such that
 \begin{equation*}
  f(\xi) = g \big(T(\xi)\big),\ \text{ where }\ T(\xi)\defeq(\xi, \adj_2 \xi,\dots,\adj_{\min\{k,n\}} \xi),
 \end{equation*}
 $\adj_s \xi$ denotes the matrix of all $s\times s$ minors of the matrix $\xi\in\mathbb R^{k\times n}$ with $2\leq s \leq \min \{k,n\}$, and
 \begin{equation*}
  \tau(k,n)\defeq \sum_{s=1}^{\min\{k,n\}}\binom{k}{s}\,\binom{n}{s}.
 \end{equation*}
\end{definition}
Let $\Omega\subset \mathbb R^n$ be a bounded Lipschitz domain and for $p>1$ define the functional 
\begin{equation*}
 F:W^{1,p}(\Omega;\mathbb R^k)\strongly \mathbb R,\ \text{ by }\ F(u)\defeq \int_\Omega W(x,u,\nabla u)\;dx,
\end{equation*}
where $W(x,\cdot,\xi)$ is a.e. convex and $G$-invariant and $W(x,s,\cdot)$ is a.e. polyconvex and $G$-invariant (one should also assume some integrability and continuity on $W$, for example that it is a so-called normal integrand). One would like to prove that if $F$ possesses a minimizer in some closed and convex set $C\subseteq W^{1,p}(\Omega;\mathbb R^k)$ (incorporating the boundary conditions), then $F$ possesses a $G$-invariant minimizer.

Arguing as in the proof of the previous theorem, we obtain the action of $G$ on $C$ is continuous and thus from Lemma \ref{inC}, the $G$-average of the minimizer $u$ satisfies $u_G\in C$. Thus it is left to show that $F(u_G)\leq F(u)$ and for that one would hope that polyconvexity allows one to use Jensen's inequality pointwise. However for that one would need to prove that
\begin{equation*}
 \adj_s \nabla u_G=\adj_s \int_G (g\cdot \nabla u)\,g\; d\theta(g)= \int_G \adj_s\big((g\cdot \nabla u)\,g\big)\; d\theta(g),
\end{equation*}
but this does not work since integrals do not respect multiplication. So a way out would be either to prove an assertion in the spirit of part $2.$ of Theorem \ref{th1}, or, assuming that more is known for the minimizer $u$, to provide a more explicit description of its $G$-average.

\def\cprime{$'$}


\begin{thebibliography}{10}

\bibitem{Alexandrov1962}
A.~D. Alexandrov.
\newblock A characteristic property of spheres.
\newblock {\em Ann. Mat. Pura Appl. (4)}, 58:303--315, 1962.

\bibitem{Ball1976}
J.~M. Ball.
\newblock Convexity conditions and existence theorems in nonlinear elasticity.
\newblock {\em Arch. Rational Mech. Anal.}, 63(4):337--403, 1976/77.

\bibitem{BelloniKawohl1999}
M.~Belloni and B.~Kawohl.
\newblock A symmetry problem related to {W}irtinger's and {P}oincar\'{e}'s
  inequality.
\newblock {\em J. Differential Equations}, 156(1):211--218, 1999.

\bibitem{BorweinJofre1998}
J.~M. Borwein and A.~Jofr{\'e}.
\newblock A nonconvex separation property in banach spaces.
\newblock {\em Mathematical Methods of Operations Research}, 48(2):169--179,
  Nov 1998.

\bibitem{BozhkovMitidieri2007}
Y.~Bozhkov and E.~Mitidieri.
\newblock Lie symmetries and criticality of semilinear differential systems.
\newblock {\em SIGMA Symmetry Integrability Geom. Methods Appl.}, 3:Paper 053,
  17, 2007.

\bibitem{Brezis1999}
H.~Brezis.
\newblock Symmetry in nonlinear {PDE}'s.
\newblock In {\em Differential equations: {L}a {P}ietra 1996 ({F}lorence)},
  volume~65 of {\em Proc. Sympos. Pure Math.}, pages 1--12. Amer. Math. Soc.,
  Providence, RI, 1999.

\bibitem{Brock2007}
F.~Brock.
\newblock Rearrangements and applications to symmetry problems in {PDE}.
\newblock In {\em Handbook of differential equations: stationary partial
  differential equations. {V}ol. {IV}}, Handb. Differ. Equ., pages 1--60.
  Elsevier/North-Holland, Amsterdam, 2007.

\bibitem{BuscaSirakov2000}
J.~Busca and B.~Sirakov.
\newblock Symmetry results for semilinear elliptic systems in the whole space.
\newblock {\em J. Differential Equations}, 163(1):41--56, 2000.

\bibitem{Ciarlet1988}
P.~G. Ciarlet.
\newblock {\em Mathematical elasticity. {V}ol. {I}, Three-dimensional
  elasticity}, volume~20 of {\em Studies in Mathematics and its Applications}.
\newblock North-Holland Publishing Co., Amsterdam, 1988.

\bibitem{ClementSweers1989}
P.~Cl\'ement and G.~Sweers.
\newblock On subsolutions to a semilinear elliptic problem.
\newblock In {\em Recent advances in nonlinear elliptic and parabolic problems
  ({N}ancy, 1988)}, volume 208 of {\em Pitman Res. Notes Math. Ser.}, pages
  267--273. Longman Sci. Tech., Harlow, 1989.

\bibitem{DaLioSirakov2007}
F.~Da~Lio and B.~Sirakov.
\newblock Symmetry results for viscosity solutions of fully nonlinear uniformly
  elliptic equations.
\newblock {\em J. Eur. Math. Soc. (JEMS)}, 9(2):317--330, 2007.

\bibitem{Dacorogna2008}
B.~Dacorogna.
\newblock {\em Direct methods in the calculus of variations}, volume~78 of {\em
  Applied Mathematical Sciences}.
\newblock Springer, New York, second edition, 2008.

\bibitem{DacorognaGangboEtAl1992}
B.~Dacorogna, W.~Gangbo, and N.~Sub\'{\i}a.
\newblock Sur une g\'en\'eralisation de l'in\'egalit\'e de {W}irtinger.
\newblock {\em Ann. Inst. H. Poincar\'e Anal. Non Lin\'eaire}, 9(1):29--50,
  1992.

\bibitem{DAmbrosioFarinaEtAl2013}
L.~D'Ambrosio, A.~Farina, E.~Mitidieri, and J.~Serrin.
\newblock Comparison principles, uniqueness and symmetry results of solutions
  of quasilinear elliptic equations and inequalities.
\newblock {\em Nonlinear Anal.}, 90:135--158, 2013.

\bibitem{FleckingerReichel2005}
J.~Fleckinger and W.~Reichel.
\newblock Global solution branches for {$p$}-{L}aplacian boundary value
  problems.
\newblock {\em Nonlinear Anal.}, 62(1):53--70, 2005.

\bibitem{Fraenkel2000}
L.~E. Fraenkel.
\newblock {\em An introduction to maximum principles and symmetry in elliptic
  problems}, volume 128 of {\em Cambridge Tracts in Mathematics}.
\newblock Cambridge University Press, Cambridge, 2000.

\bibitem{GasinskiPapageorgiou2006}
L.~Gasi{\'n}ski and N.~S. Papageorgiou.
\newblock {\em Nonlinear analysis}, volume~9 of {\em Series in Mathematical
  Analysis and Applications}.
\newblock Chapman \& Hall/CRC, Boca Raton, FL, 2006.

\bibitem{GazzolaGrunauEtAl2010}
F.~Gazzola, H.-C. Grunau, and G.~Sweers.
\newblock {\em Polyharmonic boundary value problems}, volume 1991 of {\em
  Lecture Notes in Mathematics}.
\newblock Springer-Verlag, Berlin, 2010.

\bibitem{GidasNiNirenberg1979}
B.~Gidas, W.~M. Ni, and L.~Nirenberg.
\newblock Symmetry and related properties via the maximum principle.
\newblock {\em Comm. Math. Phys.}, 68(3):209--243, 1979.

\bibitem{GidasNiNirenberg1981}
B.~Gidas, W.~M. Ni, and L.~Nirenberg.
\newblock Symmetry of positive solutions of nonlinear elliptic equations in
  {${\bf R}^{n}$}.
\newblock In {\em Mathematical analysis and applications, {P}art {A}}, volume~7
  of {\em Adv. in Math. Suppl. Stud.}, pages 369--402. Academic Press, New
  York-London, 1981.

\bibitem{HajaiejKroemer2012}
H.~Hajaiej and S.~Kr\"omer.
\newblock A weak-strong convergence property and symmetry of minimizers of
  constrained variational problems in {$\mathbb R^N$}.
\newblock {\em J. Math. Anal. Appl.}, 389(2):915--931, 2012.

\bibitem{HerzogReichel2012}
G.~Herzog and W.~Reichel.
\newblock Symmetry of solutions for quasimonotone second-order elliptic systems
  in ordered {B}anach spaces.
\newblock {\em Math. Ann.}, 352(1):99--112, 2012.

\bibitem{HytoenenVanNeervenEtAl2016}
T.~Hyt\"onen, J.~van Neerven, M.~Veraar, and L.~Weis.
\newblock {\em Analysis in {B}anach spaces. {V}ol. {I}. {M}artingales and
  {L}ittlewood-{P}aley theory}, volume~63 of {\em Ergebnisse der Mathematik und
  ihrer Grenzgebiete. 3. Folge. A Series of Modern Surveys in Mathematics}.
\newblock Springer, Cham, 2016.

\bibitem{JarohsWeth2016}
S.~Jarohs and T.~Weth.
\newblock Symmetry via antisymmetric maximum principles in nonlocal problems of
  variable order.
\newblock {\em Ann. Mat. Pura Appl. (4)}, 195(1):273--291, 2016.

\bibitem{JeanjeanSquassina2009}
L.~Jeanjean and M.~Squassina.
\newblock Existence and symmetry of least energy solutions for a class of
  quasi-linear elliptic equations.
\newblock {\em Ann. Inst. H. Poincar\'e Anal. Non Lin\'eaire},
  26(5):1701--1716, 2009.

\bibitem{Kawohl1985}
B.~Kawohl.
\newblock {\em Rearrangements and convexity of level sets in {PDE}}, volume
  1150 of {\em Lecture Notes in Mathematics}.
\newblock Springer-Verlag, Berlin, 1985.

\bibitem{Kawohl1998}
B.~Kawohl.
\newblock Symmetry or not?
\newblock {\em Math. Intelligencer}, 20(2):16--22, 1998.

\bibitem{KawohlSweers2002}
B.~Kawohl and G.~Sweers.
\newblock Inheritance of symmetry for positive solutions of semilinear elliptic
  boundary value problems.
\newblock {\em Ann. Inst. H. Poincar\'e Anal. Non Lin\'eaire}, 19(5):705--714,
  2002.

\bibitem{Kroemer2008}
S.~Kr\"omer.
\newblock Existence and symmetry of minimizers for nonconvex radially symmetric
  variational problems.
\newblock {\em Calc. Var. Partial Differential Equations}, 32(2):219--236,
  2008.

\bibitem{Lopes1996_1}
O.~Lopes.
\newblock Radial and nonradial minimizers for some radially symmetric
  functionals.
\newblock {\em Electron. J. Differential Equations}, pages No.\ 03, approx.\ 14
  pp.\, 1996.

\bibitem{Lopes1996_2}
O.~Lopes.
\newblock Radial symmetry of minimizers for some translation and rotation
  invariant functionals.
\newblock {\em J. Differential Equations}, 124(2):378--388, 1996.

\bibitem{Maris2009}
M.~Mari\c{s}.
\newblock On the symmetry of minimizers.
\newblock {\em Arch. Ration. Mech. Anal.}, 192(2):311--330, 2009.

\bibitem{MontenegroValdinoci2012}
M.~Montenegro and E.~Valdinoci.
\newblock Minimizers of a variational problem inherit the symmetry of the
  domain.
\newblock {\em Electron. J. Differential Equations}, pages No. 145, 6, 2012.

\bibitem{Nicolaescu1988}
L.~I. Nicolaescu.
\newblock Optimal control for a nonlinear diffusion equation.
\newblock {\em Ricerche Mat.}, 37(1):3--27, 1988.

\bibitem{PacellaRamaswamy2008}
F.~Pacella and M.~Ramaswamy.
\newblock Symmetry of solutions of elliptic equations via maximum principles.
\newblock In {\em Handbook of differential equations: stationary partial
  differential equations. {V}ol. {VI}}, Handb. Differ. Equ., pages 269--312.
  Elsevier/North-Holland, Amsterdam, 2008.

\bibitem{PacellaWeth2007}
F.~Pacella and T.~Weth.
\newblock Symmetry of solutions to semilinear elliptic equations via {M}orse
  index.
\newblock {\em Proc. Amer. Math. Soc.}, 135(6):1753--1762, 2007.

\bibitem{Pohozaev1965}
S.~I. Poho\v{z}aev.
\newblock On the eigenfunctions of the equation {$\Delta u+\lambda f(u)=0$}.
\newblock {\em Dokl. Akad. Nauk SSSR}, 165:36--39, 1965.

\bibitem{PucciSerrin2007}
P.~Pucci and J.~Serrin.
\newblock {\em The maximum principle}, volume~73 of {\em Progress in Nonlinear
  Differential Equations and their Applications}.
\newblock Birkh\"auser Verlag, Basel, 2007.

\bibitem{Reichel1997}
W.~Reichel.
\newblock Radial symmetry for elliptic boundary-value problems on exterior
  domains.
\newblock {\em Arch. Rational Mech. Anal.}, 137(4):381--394, 1997.

\bibitem{ReichelZou2000}
W.~Reichel and H.~Zou.
\newblock Non-existence results for semilinear cooperative elliptic systems via
  moving spheres.
\newblock {\em J. Differential Equations}, 161(1):219--243, 2000.

\bibitem{Serrin1971}
J.~Serrin.
\newblock A symmetry problem in potential theory.
\newblock {\em Arch. Rational Mech. Anal.}, 43:304--318, 1971.

\bibitem{SerrinZou1999}
J.~Serrin and H.~Zou.
\newblock Symmetry of ground states of quasilinear elliptic equations.
\newblock {\em Arch. Ration. Mech. Anal.}, 148(4):265--290, 1999.

\bibitem{Sirakov2001}
B.~Sirakov.
\newblock Symmetry for exterior elliptic problems and two conjectures in
  potential theory.
\newblock {\em Ann. Inst. H. Poincar\'e Anal. Non Lin\'eaire}, 18(2):135--156,
  2001.

\bibitem{Struwe2008}
M.~Struwe.
\newblock {\em Variational methods}, volume~34 of {\em Ergebnisse der
  Mathematik und ihrer Grenzgebiete. 3. Folge.}
\newblock Springer-Verlag, Berlin, fourth edition, 2008.
\newblock Applications to nonlinear partial differential equations and
  Hamiltonian systems.

\bibitem{Troy1981}
W.~C. Troy.
\newblock Symmetry properties in systems of semilinear elliptic equations.
\newblock {\em J. Differential Equations}, 42(3):400--413, 1981.

\bibitem{VanSchaftingen2005}
J.~Van~Schaftingen.
\newblock Symmetrization and minimax principles.
\newblock {\em Commun. Contemp. Math.}, 7(4):463--481, 2005.

\bibitem{Weth2010}
T.~Weth.
\newblock Symmetry of solutions to variational problems for nonlinear elliptic
  equations via reflection methods.
\newblock {\em Jahresbericht der Deutschen Mathematiker-Vereinigung},
  112(3):119--158, Sep 2010.

\end{thebibliography}
\end{document}